\documentclass[11pt]{amsart}

\usepackage[T1]{fontenc}
\usepackage[utf8]{inputenc}
\usepackage{lmodern}
\usepackage{amsmath,amssymb,amsthm,mathtools}
\usepackage{enumitem}
\usepackage{hyperref}
\usepackage[margin=1.15in]{geometry}
\usepackage[dvipsnames]{xcolor}
\hypersetup{
    colorlinks=true,
    linkcolor=blue,
    citecolor=blue,
    urlcolor=blue,
    filecolor=blue,
    anchorcolor=blue
}
\newtheorem{lemma}{Lemma}

\newtheorem{maintheorem}{Theorem}

\newtheoremstyle{remarkbold}
  {6pt plus 2pt minus 2pt}
  {6pt plus 2pt minus 2pt}
  {\normalfont}
  {}
  {\bfseries}
  {.}
  {.5em}
  {}

\theoremstyle{remarkbold}
\newtheorem{definition}{Definition}

\newtheorem{remark}{Remark}

\DeclareMathOperator{\dist}{dist}

\DeclareMathOperator{\Gr}{Gr}

\newcommand{\F}{\mathbb F}
\newcommand{\Torus}{\mathbb T_{\mathbb F}}
\newcommand{\LL}{\mathcal L}
\newcommand{\eps}{\varepsilon}

\title[Bishop--Phelps--Bollob\'{a}s Property for operators into $C(L)$]{The Bishop--Phelps--Bollob\'as Property for Extremally Disconnected Ranges}

\author[Amrutam]{Tattwamasi Amrutam}\address[T. Amrutam]{Institute of Mathematics, Polish Academy of Sciences,
ul. \'Sniadeckich 8, 00-656 Warszawa, Poland}\email{tattwamasiamrutam@gmail.com}
\thanks{T.A. is supported by  National Science Centre, Poland Sonata, grant number
2025/59/D/ST1/03117.}
\author[Dey]{Priyadarshi Dey}\address[P. Dey]{Millsaps College, Jackson, Mississippi, USA}\email{priyadarshid4@gmail.com, deyp@millsaps.edu}
\author[Liu]{Chunlin Liu\textsuperscript{*}}
\address[C. Liu]{School of Mathematical Sciences, Dalian University of Technology,
Dalian 116024, P.R. China, and Institute of Mathematics, Polish Academy of Sciences,
ul. \'Sniadeckich 8, 00-656 Warszawa, Poland}
\email{chunlinliu@mail.ustc.edu.cn}
\author[Monika]{Monika}\address[Monika]{Hampton University, Hampton, VA, USA}\email{Monika.ganit@gmail.com}
\thanks{C. Liu   was supported by the National Natural Science Foundation of China
(Grant No.~12601341), the Doctoral Research Start-up Fund Project of Liaoning Province (Project No.~2026-BS-0032),  the
		Postdoctoral Fellowship Program and China Postdoctoral
		Science Foundation under Grant Number BX20250067, and the China Postdoctoral Science
		Foundation under Grant Number 2025M773074.}
\thanks{\textsuperscript{*}Corresponding author.}
\begin{document}

\begin{abstract}
Let $\F$ be either $\mathbb R$ or $\mathbb C$, and let $L$ be a nonempty
extremally disconnected compact Hausdorff space. We show that
$(X,C(L,\F))$ has the Bishop--Phelps--Bollob\'as property with
$\eta(\varepsilon)=\varepsilon^2/2$ for every Banach space $X$ over
$\F$. For each fixed $L$ and either scalar field, this bound is optimal
uniformly over $X$.
\end{abstract}

\maketitle
\section{Introduction}
The Bishop--Phelps theorem~\cite{BP1961} states that for every Banach
space $X$, the norm-attaining functionals are dense in $X^*$.
Bollob\'as~\cite{Bollobas1970} showed that one can simultaneously
approximate the functional and the point at which it almost attains
its norm. Let $X$ be a real or complex Banach space and let $0<\eps<1$.
By \cite[Corollary~2.4(a)]{ChicaEtAl2014}, if $x\in S_X$ and $x^*\in S_{X^*}$
satisfy $\operatorname{Re} x^*(x)>1-\eps^2/2$, then there exist
$y\in S_X$ and $y^*\in S_{X^*}$ such that $y^*(y)=1$,
$\|y-x\|<\eps$, and $\|y^*-x^*\|<\eps$.

For arbitrary Banach spaces $X$ and $Y$, the norm-attaining operators
from $X$ to $Y$ need not be dense~\cite{Lindenstrauss1963}. This led
Acosta, Aron, Garc\'ia, and Maestre~\cite{Acosta2008} to introduce the
following property.

\begin{definition}[Bishop--Phelps--Bollob\'as property]
Let $X$ and $Y$ be Banach spaces over the same scalar field. The pair
$(X,Y)$ has the Bishop--Phelps--Bollob\'as property (BPBp, for short)
if for every $\eps\in(0,1)$ there exists $\eta(\eps)>0$ such that
whenever a norm-one operator $T\colon X\to Y$ and a unit vector
$x_0\in X$ satisfy $\|Tx_0\|>1-\eta(\eps)$, there exist a norm-one
operator $S\colon X\to Y$ and a unit vector $u_0\in X$ such that
$\|Su_0\|=1$, $\|u_0-x_0\|<\eps$, and $\|S-T\|<\eps$.
\end{definition}

If $X$ is uniformly convex, then $(X,Y)$ has the BPBp for every Banach
space $Y$~\cite[Theorem~3.1]{KimLee2014}. Choi, Kim, Lee, and
Mart\'in~\cite{CKLM2014} obtained partial results for $X=L_1(\mu)$.
See \cite{DGMR2022BPB} for a survey.

Let $K$ and $S$ be compact Hausdorff spaces. For real scalars, Johnson
and Wolfe~\cite{JohnsonWolfe1979} proved that the norm-attaining
operators from $C(K)$ to $C(S)$ are dense, and
Acosta et al.~\cite[Theorem~2.5]{AcostaBecerraEtAl2014} proved the BPBp
for this pair. For complex scalars, Garc\'ia, Maestre, and
Rodr\'iguez-Vidanes~\cite[Theorem~1.4]{MaestreGarciaRodriguezVidanes2026}
proved the corresponding density theorem, and Dantas and del
R\'io~\cite[Theorem~A]{DantasdelRio2026} proved the BPBp.

For an arbitrary real Banach space $X$, Aron, Cascales, and
Kozhushkina~\cite[Theorem~2.4]{AronCascalesKozhushkina2011} proved a
Bishop--Phelps--Bollob\'as theorem for Asplund operators from $X$ into
$C_0(L)$, where $L$ is locally compact Hausdorff. In particular,
$(X,C_0(L))$ has the BPBp when $X$ is Asplund
\cite[Corollary~2.6]{AronCascalesKozhushkina2011}. Cascales, Guirao, and
Kadets~\cite[Theorem~3.6]{CascalesGuiraoKadets2013} proved a corresponding
result for Asplund operators into complex uniform algebras, including
the BPBp for compact operators into complex $C(K)$. Cheng, Cheng, Xu,
Zhang, and Zheng~\cite[Theorem~1.2]{ChengEtAl2020} obtained the bound
$\eta(\eps)=\eps^2/2$ for real Asplund operators into $C(K)$.
The class of Asplund operators includes weakly compact, compact, and
finite-rank operators.

If $K$ is scattered, then $C(K)$ has property $\beta$, so $(X,C(K))$
has the BPBp for every Banach space $X$ by
\cite[Theorem~2.2]{Acosta2008}. Acosta et
al.~\cite[Theorem~3.3]{AcostaBecerraEtAl2014} also proved the BPBp
for compact operators from $C_0(L)$ into a uniformly convex Banach space
in both scalar fields. Dantas et al.~\cite{DGMM2016} gave conditions
under which the BPBp for compact operators on sequence spaces passes
to function spaces.

We consider extremally disconnected compact Hausdorff spaces. For
example, the Gelfand spectrum of $L^\infty[0,1]$ is the Stone space of
the Lebesgue measure algebra and is extremally disconnected with no
isolated points.

\begin{maintheorem}
\label{thm:main}
Let $\F$ be either $\mathbb R$ or $\mathbb C$, let $X$ be a Banach
space over $\F$, and let $L$ be a nonempty extremally disconnected
compact Hausdorff space. Let $\varepsilon\in(0,1)$ and
$T\in\LL(X,C(L,\F))$ with $\|T\|\leq1$. If $x_0\in S_X$ satisfies
$\|Tx_0\|>1-\varepsilon^2/2$, then there exist $u_0\in S_X$ and
$S\in\LL(X,C(L,\F))$ such that $\|S\|=\|Su_0\|=1$,
$\|u_0-x_0\|<\varepsilon$, and $\|S-T\|<\varepsilon$.

In particular, $(X,C(L,\F))$ has the BPBp. For each fixed $L$ and
either scalar field, $\eta(\varepsilon)=\varepsilon^2/2$ cannot be
increased uniformly over all Banach spaces $X$.
\end{maintheorem}

The results for $(C(K),C(L))$ above allow arbitrary compact Hausdorff
spaces $K$ and $L$. When $L$ is extremally disconnected,
Theorem~\ref{thm:main} applies to all bounded operators from an
arbitrary Banach space $X$ into $C(L)$.
\subsection*{Acknowledgments}
T.A. thanks Grigor Sargsyan and Piotr Koszmider for useful discussions. We also thank the anonymous referee for many corrections and useful remarks that improved the readability of the paper. 
\section{Notation and Preliminaries}\label{sec:pre}

Throughout, all Banach spaces and operators are over a fixed field $\F$,
either $\mathbb R$ or $\mathbb C$. Write
$\Torus=\{\lambda\in\F:|\lambda|=1\}$. For a Banach space $X$,
write $B_X$ and $S_X$ for its closed unit ball and unit sphere and $X^*$
for its dual. If $L$ is a nonempty compact Hausdorff space, write
$C(L)=C(L,\F)$ for the Banach space of continuous $\F$-valued functions on
$L$ with the supremum norm. We denote by $\LL(X,C(L))$ the space of bounded
linear operators from $X$ to $C(L)$ and by $\mathbb Q_{>0}$ the set of
positive rational numbers.

Recall that a compact Hausdorff space $L$ is \emph{extremally disconnected}
if the closure of every open subset of $L$ is open. Equivalently, $L$ has
a clopen base and any two disjoint open subsets have disjoint closures.
Examples include the Stone--\v{C}ech compactification $\beta\Gamma$ of a
discrete set $\Gamma$, the Gelfand spectra of commutative von Neumann
algebras, and the Stone spaces of complete Boolean algebras.

We use the clopen base to apply Gleason's theorem on clopen subspaces of $L$.

\begin{lemma}\label{lem:clopen-base}
Let $L$ be an extremally disconnected compact Hausdorff space.
\begin{enumerate}[label=\textup{(\roman*)},leftmargin=*]
\item $L$ has a base of clopen neighborhoods.
\item Every clopen subspace of $L$ is extremally disconnected compact
Hausdorff.
\end{enumerate}
\end{lemma}

\begin{proof}
For \textup{(i)}, let $p\in U$ with $U$ open in $L$. By regularity, choose
an open set $W$ with $p\in W\subseteq\overline W\subseteq U$. Since $L$
is extremally disconnected, $\overline W$ is a clopen neighborhood of $p$
contained in $U$.

For \textup{(ii)}, let $C$ be a clopen subset of $L$. Then $C$ is compact Hausdorff. If $U$ is open in $C$, then $U$ is open in $L$ and its closure
in $C$ is $C\cap\overline U$, which is open in $C$.
\end{proof}

\begin{definition}[Admissible scalar set]
Let $X$ be a Banach space over $\F$, let $L$ be a nonempty compact
Hausdorff space, and let $T\in\LL(X,C(L))$ satisfy $\|T\|\leq1$. Fix
$u\in S_X$ and $\delta>0$. For $s\in L$, write $T_s(x)=(Tx)(s)$ for
$x\in X$. The \emph{admissible scalar set} at $s$ is
$A_{\delta,u,T}(s)=\{a(u):a\in B_{X^*},\ \|a-T_s\|\leq\delta\}$.
\end{definition}

For such $T$, we have $T_s\in B_{X^*}$ for every $s\in L$, and
$s\mapsto T_s$ is weak-star continuous since $s\mapsto (Tx)(s)$ is
continuous for each $x\in X$. The next lemma lifts a continuous selection
of the admissible scalar sets to an operator.

\begin{lemma}\label{lem:scalar-to-operator}
Let $L$ be an extremally disconnected compact Hausdorff space. Fix
$T\in\LL(X,C(L))$ with $\|T\|\leq1$, $u\in S_X$, and $\delta>0$.
Suppose that $g\in C(L)$ satisfies $g(s)\in A_{\delta,u,T}(s)$ for all
$s\in L$. Then there is an operator $S:X\to C(L)$ with $Su=g$,
$\|S\|\leq1$, and $\|S-T\|\leq\delta$.
\end{lemma}

\begin{proof}
Endow $Z=X\times X$ with the norm $\|(y,x)\|_Z=\|y\|+\delta\|x\|$.
The subspace $E=\{(x+\alpha u,x):x\in X,\ \alpha\in\F\}$ is closed
because $(y,x)\in E$ if and only if $y-x\in\operatorname{span}\{u\}$.
Define $F:E\to C(L)$ by $F(x+\alpha u,x)=Tx+\alpha g$. This is
well-defined since the second coordinate determines $x$ and the first
then determines $\alpha$.

Fix $s\in L$. Since $g(s)\in A_{\delta,u,T}(s)$, choose
$a_s\in B_{X^*}$ with $a_s(u)=g(s)$ and $\|a_s-T_s\|\leq\delta$.
Then $F(x+\alpha u,x)(s)=a_s(x+\alpha u)+(T_s-a_s)(x)$. Hence
$|F(x+\alpha u,x)(s)|\leq\|x+\alpha u\|+\delta\|x\|
=\|(x+\alpha u,x)\|_Z$.
Taking the supremum over $s\in L$ gives $\|F\|\leq1$. Since $L$ is
extremally disconnected, $C(L)$ is $1$-injective by Goodner, Kelley,
and Nachbin~\cite{Goodner1950,Kelley1952,Nachbin1950} in the real case
and Hasumi~\cite{Hasumi1958} in the complex case. Thus $F$ extends to
$\widetilde F:Z\to C(L)$ with $\|\widetilde F\|\leq1$.
Let $S:X\to C(L)$ be given by $Sx=\widetilde F(x,0)$. Then
$\|Sx\|\leq\|(x,0)\|_Z=\|x\|$, and $Su=F(u,0)=g$ since $(u,0)\in E$.
Also, $(x,x)\in E$, so $Tx=\widetilde F(x,x)=Sx+\widetilde F(0,x)$.
Thus $\|(T-S)x\|\leq\|(0,x)\|_Z=\delta\|x\|$, as required.
\end{proof}

To construct the selection, we use compactness of the graph of the
admissible scalar sets.

\begin{lemma}\label{lem:graph-compact}
Let $L$ be a nonempty compact Hausdorff space and let $T:X\to C(L)$
satisfy $\|T\|\leq1$. Fix $u\in S_X$ and $\delta>0$. Then the graph
$\Gr(A_{\delta,u,T})=\{(s,z)\in L\times\F:z\in A_{\delta,u,T}(s)\}$
is compact, and each fibre $A_{\delta,u,T}(s)$ is a nonempty compact
convex subset of the closed unit disk.
\end{lemma}

\begin{proof}
Each fibre contains $T_s(u)$ since $T_s\in B_{X^*}$. It is contained in
the closed unit disk and is convex, since it is the image of the
intersection of two norm balls under evaluation at $u$.

Endow $B_{X^*}$ with the weak-star topology and let
$\mathcal E=\{(s,a)\in L\times B_{X^*}:\|a-T_s\|\leq\delta\}$.
The map $(s,a)\mapsto a-T_s$ is weak-star continuous. Since the dual norm
is weak-star lower semicontinuous, $\mathcal E$ is closed in the compact
space $L\times B_{X^*}$. The graph $\Gr(A_{\delta,u,T})$ is the image of
$\mathcal E$ under the continuous map $(s,a)\mapsto(s,a(u))$, so it is
compact. Each fibre $A_{\delta,u,T}(s)$ is homeomorphic to the closed
subset $\Gr(A_{\delta,u,T})\cap(\{s\}\times\F)$ of the graph, and hence
is compact.
\end{proof}

\section{Proof of the main theorem}\label{sec:main}

Fix a Banach space $X$, a nonempty extremally disconnected compact
Hausdorff space $L$, and an operator $T:X\to C(L)$ with $\|T\|\leq1$.
To construct a selection of the admissible scalar sets, we use distances
from $T_s$ to sets of functionals whose value at a given unit vector is
close to a prescribed scalar.

\begin{definition}
Let $X$ be a Banach space over $\mathbb F$, let $L$ be a nonempty
extremally disconnected compact Hausdorff space, and let $T:X\to C(L)$
satisfy $\|T\|\leq1$. For a nonzero closed subspace $E\subseteq X$,
$v\in S_E$, $\zeta\in\mathbb T_{\mathbb F}$, and $q\in\mathbb Q_{>0}$,
put $F_E(v,\zeta,q)=\{a\in B_{E^*}:|a(v)-\zeta|\leq q\}$.
For $s\in L$, define $d^E_{v,\zeta,q}(s)=\dist(T_s|_E,F_E(v,\zeta,q))$,
where $T_s(x)=(Tx)(s)$ for $x\in X$ and the distance is computed in
the norm of $E^*$.
\end{definition}

Fix a countable dense subset $P\subseteq\mathbb T_{\mathbb F}$.
For $E,v,\zeta,q$ as above, the Hahn--Banach theorem gives
$a\in S_{E^*}$ such that $a(v)=\zeta$. Thus $F_E(v,\zeta,q)$ is
nonempty. It is weak-star closed in $B_{E^*}$, hence weak-star compact.
Since $\|T_s|_E\|\leq1$, we have $0\leq d^E_{v,\zeta,q}(s)\leq2$ for
every $s\in L$. We first express the distances in $X^*$ in terms of
finite-dimensional subspaces.

\begin{lemma}\label{lem:witness}
Let $v\in S_X$, $\zeta\in P$, and $q\in\mathbb Q_{>0}$. For every $s\in L$,
\[
d^X_{v,\zeta,q}(s)=
\sup_{\substack{E\subseteq X\text{ finite-dimensional}\\v\in E}}
d^E_{v,\zeta,q}(s).
\]
For each finite-dimensional $E$ containing $v$, the function
$d^E_{v,\zeta,q}$ is continuous. In particular, $d^X_{v,\zeta,q}$ is
lower semicontinuous.
\end{lemma}

\begin{proof}
If $a\in F_X(v,\zeta,q)$, then $a|_E\in F_E(v,\zeta,q)$ and
$\|a|_E-T_s|_E\|\leq\|a-T_s\|$. Thus $d^E_{v,\zeta,q}(s)\leq
d^X_{v,\zeta,q}(s)$.

Fix $s\in L$ and let $r$ be the supremum in the statement. For each
$x\in X$, let $K_x=\{a\in F_X(v,\zeta,q):|a(x)-T_s(x)|\leq r\|x\|\}$.
These sets are weak-star closed in the weak-star compact set
$F_X(v,\zeta,q)$. Let $x_1,\ldots,x_m\in X$ and let
$E=\operatorname{span}\{v,x_1,\ldots,x_m\}$. Since $F_E(v,\zeta,q)$ is
norm compact and $d^E_{v,\zeta,q}(s)\leq r$, there is
$a_E\in F_E(v,\zeta,q)$ with $\|a_E-T_s|_E\|\leq r$. By the
Hahn--Banach theorem, $a_E$ extends to $a\in B_{X^*}$. Then
$a\in F_X(v,\zeta,q)$ and $a\in K_{x_j}$ for $1\leq j\leq m$.
Thus the family $(K_x)_{x\in X}$ has the finite intersection property.
By compactness, there is $a\in\bigcap_{x\in X}K_x$. Hence
$\|a-T_s\|\leq r$, so $d^X_{v,\zeta,q}(s)\leq r$.

If $E$ is finite-dimensional, evaluation on a basis of $E$ shows that
$s\mapsto T_s|_E$ is norm continuous. The distance to a fixed
nonempty set is $1$-Lipschitz, so $d^E_{v,\zeta,q}$ is continuous.
Thus $d^X_{v,\zeta,q}$ is a supremum of continuous functions.
\end{proof}

We apply the next lemma to the family of distances in
Lemma~\ref{lem:witness}.

\begin{lemma}\label{lem:local}
Let $I$ be nonempty and let $h_i\in C(L,\mathbb R)$ for $i\in I$.
Suppose $h=\sup_{i\in I}h_i$ is bounded. If $C\subseteq L$ is nonempty
and clopen and $\tau>0$, then there exist $i\in I$ and a nonempty
clopen set $C'\subseteq C$ such that $\operatorname{osc}_{C'}h<\tau$
and $0\leq h(s)-h_i(s)<\tau$ for every $s\in C'$, where
$\operatorname{osc}_{C'}h=\sup_{s\in C'}h(s)-\inf_{s\in C'}h(s)$.
\end{lemma}

\begin{proof}
Let $m=\sup_{s\in C}h(s)$. Choose $t\in C$ and $i\in I$ such that
$h_i(t)>m-\tau$. By Lemma~\ref{lem:clopen-base}, the nonempty open set
$\{s\in C:h_i(s)>m-\tau\}$ contains a nonempty clopen set $C'$. Since
$C'$ is compact, $b=\min_{s\in C'}h_i(s)>m-\tau$. For $s\in C'$
we have $b\leq h_i(s)\leq h(s)\leq m$. Thus both
$\operatorname{osc}_{C'}h$ and $h(s)-h_i(s)$ are bounded above by
$m-b<\tau$.
\end{proof}

Applying the preceding lemmas successively gives a separable subspace
and a point at which the distances computed on that subspace agree
with the corresponding distances computed on $X$.

\begin{lemma}\label{lem:reduction}
Let $Z\subseteq X$ be a nonzero separable closed subspace and let
$O\subseteq L$ be nonempty and open. There exist a separable closed
subspace $Y\subseteq X$ containing $Z$, a countable dense subset
$D\subseteq S_Y$, and a point $p\in O$ such that, for every
$(v,\zeta,q)\in D\times P\times\mathbb Q_{>0}$, the function
$d^X_{v,\zeta,q}$ is continuous at $p$ and
$d^Y_{v,\zeta,q}(p)=d^X_{v,\zeta,q}(p)$.
\end{lemma}

\begin{proof}
Choose a nonempty clopen set $C_0\subseteq O$ and let $Y_0=Z$.
After choosing $Y_i$, choose a countable dense subset $D_i\subseteq
S_{Y_i}$ and enumerate $D_i\times P\times\mathbb Q_{>0}=
\{t(i,k):k\geq1\}$. For $t=(v,\zeta,q)$, write $d^E_t=d^E_{v,\zeta,q}$
whenever $v\in E$.

Suppose $Y_0,\ldots,Y_{n-1}$, their enumerations, and $C_{n-1}$ have
been chosen. Let $\mathcal T_n=\{t(i,k):0\leq i<n,\ 1\leq k\leq n\}$.
Starting with $C_{n-1}$, apply Lemmas~\ref{lem:witness}
and~\ref{lem:local} successively to the triples in $\mathcal T_n$ with
$\tau=2^{-n}$. For each $t=(v,\zeta,q)$ this gives a finite-dimensional
subspace $E_{n,t}$ containing $v$ and a smaller nonempty clopen set.
Passing to a smaller set preserves the bounds from the preceding
steps. Thus the final nonempty clopen set $C_n\subseteq C_{n-1}$
satisfies
\[
\operatorname{osc}_{C_n}d^X_t<2^{-n}
~\text{and}~
0\leq d^X_t(s)-d^{E_{n,t}}_t(s)<2^{-n}
\]
for every $t\in\mathcal T_n$ and $s\in C_n$. Define the separable
subspace $Y_n=\overline{\operatorname{span}
(Y_{n-1}\cup\bigcup_{t\in\mathcal T_n}E_{n,t})}$, and choose $D_n$ and
its enumeration as above.

Let $Y=\overline{\bigcup_{n\geq0}Y_n}$, $D=\bigcup_{n\geq0}D_n$, and
$H=\bigcap_{n\geq0}C_n$. Then $Y$ is separable and contains $Z$, and
$D$ is countable. To see that $D$ is dense in $S_Y$, fix $y\in S_Y$
and $\varepsilon>0$. Choose $n$ and a nonzero $w\in Y_n$ with
$\|w-y\|<\varepsilon/4$. Then $\|w/\|w\|-y\|\leq2\|w-y\|<
\varepsilon/2$. Choose $v\in D_n$ with $\|v-w/\|w\|\|<\varepsilon/2$.
Thus $\|v-y\|<\varepsilon$.

Since the $C_n$ are nonempty, compact, and nested, choose $p\in H$.
Fix $t=t(i,k)$. For every $n\geq\max\{i+1,k\}$, we have
$t\in\mathcal T_n$. Since $C_n$ is an open neighborhood of $p$ and
$\operatorname{osc}_{C_n}d^X_t<2^{-n}$, the function $d^X_t$ is
continuous at $p$. Also, $E_{n,t}\subseteq Y\subseteq X$, so
restriction of functionals gives
$d^{E_{n,t}}_t(p)\leq d^Y_t(p)\leq d^X_t(p)$. Thus
$0\leq d^X_t(p)-d^Y_t(p)<2^{-n}$ for every sufficiently large $n$,
and $d^Y_t(p)=d^X_t(p)$.
\end{proof}

\begin{proof}[Proof of Theorem~\ref{thm:main}]
Let $0<\varepsilon<1$, and suppose that $T:X\to C(L)$ and $x_0\in S_X$
satisfy $\|T\|\leq1$ and $\|Tx_0\|>1-\varepsilon^2/2$. The set
$O=\{s\in L:|(Tx_0)(s)|>1-\varepsilon^2/2\}$ is nonempty and open.
Apply Lemma~\ref{lem:reduction} with $Z=\operatorname{span}\{x_0\}$ and
$O$ to obtain a separable closed subspace $Y\subseteq X$ containing $x_0$,
a countable dense subset $D\subseteq S_Y$, and a point $p\in O$.

Choose $\theta\in\Torus$ such that $\theta T_p(x_0)=|T_p(x_0)|$ and let
$f=\theta T_p|_Y\in B_{Y^*}$. Then $\operatorname{Re}f(x_0)>1-\varepsilon^2/2$. By
\cite[Corollary~2.4(a)]{ChicaEtAl2014}, there exist $u_0\in S_Y$ and
$\psi\in S_{Y^*}$ such that $\psi(u_0)=1$, $\|u_0-x_0\|<\varepsilon$,
and $\|\psi-f\|<\varepsilon$. Put $\omega=\overline{\theta}\psi$ and
$\lambda=\overline{\theta}$. Then $\omega\in S_{Y^*}$,
$\omega(u_0)=\lambda$, and $\beta:=\|\omega-T_p|_Y\|=\|\psi-f\|<\varepsilon$.
Let $\delta=(\beta+\varepsilon)/2$ and write $A(s)=A_{\delta,u_0,T}(s)$.

We use Lemma~\ref{lem:reduction} to find values in $A(s)$ close to
$\lambda$ on neighborhoods of $p$. For $n\geq1$, put $r_n=1/n$ and
choose $v_n\in D$, $\zeta_n\in P$, and $q_n\in\mathbb Q_{>0}$ such that
$\|v_n-u_0\|<r_n/8$, $|\zeta_n-\lambda|<r_n/8$, and
$r_n/2<q_n<5r_n/8$. Then
$|\omega(v_n)-\zeta_n|\leq\|v_n-u_0\|+|\zeta_n-\lambda|<q_n$,
so $\omega\in F_Y(v_n,\zeta_n,q_n)$. By Lemma~\ref{lem:reduction},
$d^X_{v_n,\zeta_n,q_n}(p)=d^Y_{v_n,\zeta_n,q_n}(p)\leq\beta<\delta$.
Since $d^X_{v_n,\zeta_n,q_n}$ is continuous at $p$, there is an open
neighborhood $V_n$ of $p$ such that $d^X_{v_n,\zeta_n,q_n}(s)<\delta$
for every $s\in V_n$. For such $s$, choose $a\in F_X(v_n,\zeta_n,q_n)$
with $\|a-T_s\|<\delta$. Then $a(u_0)\in A(s)$ and
$|a(u_0)-\lambda|\leq\|u_0-v_n\|+q_n+|\zeta_n-\lambda|<7r_n/8$.
Thus $A(s)\cap\overline B(\lambda,r_n)\ne\varnothing$ for every
$s\in V_n$.

To select these values continuously, choose a clopen neighborhood
$B_1\subseteq V_1$ of $p$. Inductively choose a clopen neighborhood
$B_{n+1}\subseteq B_n\cap V_{n+1}$ of $p$. Let
$H=\bigcap_{n\geq1}B_n$ and $Q_n=B_n\setminus B_{n+1}$. For each
nonempty $Q_n$, put
$G_n=\{(s,z)\in Q_n\times\overline B(\lambda,r_n):z\in A(s)\}$.
Since $G_n$ is closed in the graph of $A$, it is compact by
Lemma~\ref{lem:graph-compact}. As $Q_n\subseteq V_n$, the coordinate
projection $\pi_n:G_n\to Q_n$ is a continuous surjection. The set $Q_n$ is clopen,
so it is extremally disconnected and compact by
Lemma~\ref{lem:clopen-base}. By Gleason's theorem~\cite{Gleason1958},
$\pi_n$ has a continuous section $\sigma_n:Q_n\to G_n$.
Writing $\sigma_n(s)=(s,g_n(s))$, we obtain a continuous function
$g_n:Q_n\to\mathbb F$ with
$g_n(s)\in A(s)\cap\overline B(\lambda,r_n)$.

For $s\in H$, we have $s\in V_n$ for every $n$, so the closed set
$A(s)$ meets $\overline B(\lambda,r_n)$ for every $n$. Since $r_n\to0$,
we have $\lambda\in A(s)$. Define $g:L\to\mathbb F$ by
\[
  g(s)=
  \begin{cases}
    (Tu_0)(s), & s\in L\setminus B_1,\\
    g_n(s), & s\in Q_n,\\
    \lambda, & s\in H.
  \end{cases}
\]
The sets $L\setminus B_1$ and $Q_n$ are clopen, so $g$ is continuous
outside $H$. If $s\in H$ and $t\in B_N$, then either $t\in H$ or
$t\in Q_m$ for some $m\geq N$. In either case
$|g(t)-g(s)|\leq r_N$. Since $B_N$ is an open neighborhood of $s$
and $r_N\to0$, $g$ is continuous at $s$.

For $s\in L\setminus B_1$, we have $g(s)=T_s(u_0)\in A(s)$.
Together with the choices on $Q_n$ and $H$, this gives $g(s)\in A(s)$
for all $s\in L$. In particular, $\|g\|\leq1$. Since $p\in H$ and
$g(p)=\lambda$, we have $\|g\|=1$. Apply
Lemma~\ref{lem:scalar-to-operator} to obtain $S:X\to C(L)$ with
$\|S\|\leq1$, $\|S-T\|\leq\delta<\varepsilon$, and $Su_0=g$.
Since $u_0\in S_X$, we get $\|S\|=\|Su_0\|=1$.
\end{proof}

\begin{remark}\label{rem:sharpness}
For every nonempty extremally disconnected compact Hausdorff space $L$
and either scalar field, the pair $(\ell_\infty^2(\mathbb F),C(L))$ admits
no threshold larger than $\varepsilon^2/2$ for any $0<\varepsilon<1$.
We adapt \cite[Example~2.5]{ChicaEtAl2014}. Fix $0<\varepsilon<1$ and put
$x_0=(1-\varepsilon,1)\in S_{\ell_\infty^2(\mathbb F)}$. Define
$T:\ell_\infty^2(\mathbb F)\to C(L)$ by
$(T(z_1,z_2))(s)=(\varepsilon/2)z_1+(1-\varepsilon/2)z_2$
for every $s\in L$.
Then $\|T\|=1$ and $\|Tx_0\|=1-\varepsilon^2/2$.

Suppose $u_0=(u_1,u_2)\in S_{\ell_\infty^2(\mathbb F)}$ and
$S:\ell_\infty^2(\mathbb F)\to C(L)$ satisfy
$\|u_0-x_0\|<\varepsilon$ and $\|S\|=\|Su_0\|=1$.
Then $|u_1|\leq |u_1-(1-\varepsilon)|+1-\varepsilon<1$.
By compactness, choose $s_0\in L$ with $|(Su_0)(s_0)|=1$. Write
$S_{s_0}(z_1,z_2)=az_1+bz_2$. Since $|a|+|b|=\|S_{s_0}\|\leq1$,
$1=|au_1+bu_2|\leq|a||u_1|+|b|\leq1$. Since $|u_1|<1$, this forces
$a=0$ and $|b|=1$. Hence
\[
  \|S-T\|\geq\|S_{s_0}-T_{s_0}\|
   =\frac{\varepsilon}{2}
     +\left|b-\left(1-\frac{\varepsilon}{2}\right)\right|
   \geq\varepsilon.
\]
Thus $(T,x_0)$ rules out every larger threshold.
\end{remark}

\bibliographystyle{alpha}
\bibliography{ref}
\end{document}